\documentclass[final]{aupl}

\usepackage{
amsfonts,
latexsym,
amssymb,
enumerate,
verbatim,
mathrsfs,
rotate,
centernot,
color,
}

\newcommand{\labbel}[1]{\label{#1} [[{\bf #1}]]}  
\newcommand{\bibbitem}[1]{\bibitem{#1} [[{\bf #1}]]}  
\renewcommand{\labbel}{\label} \renewcommand{\bibbitem}{\bibitem}

\newcommand{\pip}{{\pi^{\scriptscriptstyle {\restriction}}}}

\DeclareMathOperator{\cf}{cf}

\newtheorem{theorem}{Theorem}[section]

\newtheorem{proposition}[theorem]{Proposition}

\newtheorem{claim}[theorem]{Claim} 
\newtheorem*{claim*}{Claim}

\newtheorem*{theorem*}{Theorem}
\newtheorem*{proposition*}{Proposition}
\newtheorem*{corollary*}{Corollary}
\newtheorem*{lemma*}{Lemma}
\newtheorem*{scholion*}{Scholion}

\theoremstyle{definition}
\newtheorem{definition}[theorem]{Definition}

\theoremstyle{remark}
\newtheorem{remark}[theorem]{Remark}

\newtheorem*{remark*}{Remark}
\newtheorem*{remarks*}{Remarks}
 
\newtheorem{example}[theorem]{Example}

\newtheorem*{observation*}{Observation}

 \allowdisplaybreaks[1]

\numberwithin{equation}{section}

\begin{document}
\title{Ordinal semigroups}

\author{Paolo Lipparini} 
\address{Dipartimento di Matematica\\Viale della  Ricerca
 Ordinale \\Universit\`a di Roma ``Tor Vergata'' 
\\I-00133 ROME ITALY (currently retired)}

\email{lipparin@axp.mat.uniroma2.it}

\urladdr{http://www.mat.uniroma2.it/\textasciitilde lipparin}

\subjclass{08A65; 20M75; 03E10}

\keywords{Partial infinitary semigroup,  ordinal semigroup, 
non-commutative semigroup, 
generalized associative laws, sequence indexed by a well-ordered set.}

\thanks{
Work performed under the auspices of G.N.S.A.G.A. Work 
partially supported by PRIN 2012 ``Logica, Modelli e Insiemi''.
The author acknowledges the MIUR Department Project awarded to the
Department of Mathematics, University of Rome Tor Vergata, CUP
E83C18000100006.}

\begin{abstract}
In a previous paper we introduced a
version of associativity for a partial infinitary operation.
We prove here that if $\gamma$ is an infinite
ordinal and some associative infinitary operation is defined 
for all  sequences indexed by ordinals $ \leq \gamma$,
 then such an operation
can be uniquely expanded to apply to every
  sequence indexed by any ordinal  of cardinality $ |\gamma |$.

In particular, if some associative operation is defined for all finite sequences 
as well as for all $ \omega$-indexed sequences, then the operation
can be uniquely expanded to apply to every sequence indexed by a 
countable ordinal.
\end{abstract}

\maketitle

\section{Introduction} \labbel{intro}

Krob  introduced   (commutative) \emph{complete  monoids}
in \cite{K}.
The notion implicitly appeared earlier in
Conway \cite{C}. See  Hebisch  and    Weinert \cite{HW} for a survey
with many applications.
In \cite{smgrnc} we generalized the notion to the
noncommutative case, dealing also with partially defined products,
in particular, encompassing former special cases treated, among others,  by
Tarski \cite{Tord},  Madevski, Trpenovski and {\v C}upona \cite{MTC}, 
Belousov and  Stojakovi\'c \cite{BS}, Perrin and Pin \cite{PP},  
 Bedon and Carton \cite{BC}.  
Our definition in \cite{smgrnc} 
is modeled after the 
 notion of a $\diamond$-semigroup in 
 Rispal and Carton \cite{RC},
but we consider products of sequences indexed
by arbitrary linearly ordered sets,
while  Rispal and Carton
deal only with 
sequences indexed by countable  scattered  linear orders. 
See also Bedon and Rispal \cite{BR}. 
In detail  
 a (partial) \emph{infinitary semigroup}  (Definition \ref{pnc} below)
 is a structure with a partial ``product operation''  defined on certain sequences
indexed by a linearly ordered set. Such an operation is  required  to satisfy a general 
form of associativity.
 
In the present paper we deal with the special case when
the index sets are ordinals, equivalently, well-ordered sets.
We prove the extension theorem mentioned in the abstract,
see Theorem \ref{ordg} for explicit details.
Since the statement of Theorem \ref{ordg} is quite involved, the reader might wish
to consider  the particular case
$\gamma= \omega $ first. In this case the content of the theorem
is the following. If we know every finite product,
as well as every product
of sequences of order-type $ \omega$ (the order-type of the naturals)
and the properties (Ord) and (U) we will
introduce  in Definitions \ref{pnc} and \ref{pord}  are satisfied relative
 to such sequences, then
 \emph{all} products of sequences 
indexed by any countable ordinal are uniquely determined, under the request that
the properties continue to be satisfied (this is the easier part)
and the request can be actually fulfilled (the harder part).

The general case of our main result starts with an
 infinite ordinal $\gamma$
and shows that the  operation can be uniquely lifted
up to 
 all sequences indexed by any ordinal $\delta$ 
having the same cardinality  of $\gamma$.
The definition of the additional products is rather easy,
but in general they depend on the choice of cofinal sequences on the 
ordinals; delicate details have to be discussed
in order to show that the definition does not depend
on any particular choice.

\section{Basic notions} \labbel{basic} 

\begin{definition} \labbel{pnc}
\cite{smgrnc}  A (partial)  \emph{infinitary semigroup}
 is a nonempty class $S$ 
together with a class function $ \prod$
whose codomain is $S$ itself,
whose domain
 consists of a class of linearly-indexed sequences
 of elements of $S$ and which satisfies the following properties:
  \begin{enumerate}
\item[(U)]
\emph{If $I= \{ i  \} $ has one element,
then 
$\prod _{i \in I} a_i $ is defined 
and is equal to $a_i$,}
   \item[(N)]
\emph{Whenever $\prod _{i \in I} a_i $ is defined
and $\pi:I \to J$ is a surjective order preserving map,
then all products in the following equation are defined
\begin{equation*}\labbel{N}
\prod _{i \in I} a_i = \prod _{j\in J} \prod _{\pi(i) =j }  a_i
  \end{equation*}     
and equality actually holds.}
   \end{enumerate} 
where $ \prod _{i \in I} a_i $ denotes 
the image of $(a_i) _{i \in I} $  under $\prod$, when defined,
and $ \prod _{\pi(i) =j }  a_i$
means $ \prod _{i \in I_j }  a_i$, where
$I_j  = \{ i \in I \mid \pi(i) =j \}  $ is given the order 
induced by the order on $I$. 

If $I= \{ 0, 1 \} $, we will write $a_0 * a_1$
in place of  $ \prod _{i \in I} a_i $. 
As usual, when there is no risk of ambiguity, we will write
$S$ in place of $(S, \prod)$. See \cite{smgrnc}
for further details, properties and many examples. 
\end{definition}

Just a small comment about the expression ``class''.
We have stated Definition \ref{pnc} in terms of classes since many interesting
examples are proper classes,  compare Example \ref{ex} below.
The reader not interested in foundational
issues might just consider Definition \ref{pnc} as always restricted to sets,
or simply consider ``class'' as a synonymous of ``collection''.   
  
In this note we will be interested in the case when
$\prod _{i \in I} a_i $ is defined only for (certain)
well-ordered index sets $I$. See e.~g. \cite{bach,J} for 
informations about well-ordered sets and ordinals.
 Since a well-order is
order isomorphic to a unique ordinal, it is enough
to define
$\prod _{ \alpha < \delta } a_\alpha  $,
for $\delta$ an ordinal (or say that the product is not defined).
This observation leads to the following
variation of Definition \ref{pnc}. We denote intervals of ordinals in the usual way,
e.~g., $[ \alpha, \beta )= \{ \, \gamma   \mid 
 \gamma \text{ is an ordinal and } \alpha \leq \gamma < \beta   \,\}$.

\begin{definition} \labbel{pord}   
A  \emph{partial ordinal semigroup} 
 is a nonempty class $S$ together
 with a class function $ \prod$
whose codomain is $S$ and
whose domain
 consists of a class of nonempty ordinal-indexed sequences
of elements of $S$,
which satisfies (U) from Definition \ref{pnc} and 
\begin{enumerate}
   \item[(Ord)]
\emph{
Assume that  $\prod _{ \alpha  <\delta } a_\alpha  $ is defined,
 $\pi: \delta  \to \eta$ is a surjective  order preserving map  and,
 for every $j < \eta$,  
let $I_j = \{ \alpha < \delta \mid \pi ( \alpha ) = j \}
= [\beta_j, \beta_j + \varepsilon_j )  $, for appropriate
uniquely determined $\beta_j $ and $  \varepsilon_j$.
Whenever the above assumptions are met,
we require that  
all the products  in
the following equation are defined
\begin{equation*}
\prod _{ \alpha  < \delta } a_ \alpha  = \prod _{j < \eta} 
\prod _{ \gamma < \varepsilon_j } a _{\beta_j + \gamma }
  \end{equation*}     
and equality actually holds.
}
   \end{enumerate} 

If $\prod _{ \alpha  <\delta } a_\alpha  $ is defined 
for every ordinal $\delta$, we will speak of an \emph{ordinal semigroup}.  
If $\delta$ is an ordinal, a \emph{${<}\gamma$-semigroup} 
(a  \emph{${\leq}\gamma$-semigroup})
 is a partial ordinal semigroup such that 
  $\prod _{ \alpha  < \delta } a_\alpha  $ is defined for every 
nonzero ordinal $\delta <  \gamma $ ($\delta \leq \gamma $) and,
just to formally keep uniqueness in the following results, is never 
defined for larger ordinals. 
 \end{definition}

\begin{example} \labbel{ex}  
 The class of the ordinals 
with transfinite sum is an ordinal semigroup; see
clause (4) in \cite[p. 51]{bach}.
The class of ordinal-indexed transfinite strings
with elements from
some nonempty fixed set $X$
is an ordinal semigroup; see
Karp \cite[Chapter II]{Karp}.
In both cases, if we restrict ourselves to 
ordinals (words of length) $< \lambda $, for $\lambda$ an infinite
regular cardinal, we get a
${<}\lambda$-semigroup. 
A general theory for $ \omega_1$-semigroups
(${<} \omega_1$-semigroups in the present terminology)
appears in \cite{BR}. 
As another example,
$ \omega$-semigroups in the sense
from  \cite[4.1]{PP} can be seen as 
a special class of partial ordinal semigroups
in the present terminology. See 
\cite[Example 4.3]{smgrnc} for details.
 \end{example}

\begin{remark} \labbel{strfrm}    
(a) Strictly formally, a partial ordinal semigroup in the sense
of Definition \ref{pord} is not an infinitary semigroup
in the sense of Definition \ref{pnc}, since we may take $\pi$
to be an order preserving  bijection from $\delta$ to some well-ordered
set $J$ isomorphic to $\delta$, but not an ordinal;
in this case $\prod _{j \in J} a_ { \pi ^{-1} (j) }    $ is not defined, according
to Definition \ref{pord}, so that (N) would fail.    
Of course, this is just hair-splitting pedantry: given a partial ordinal semigroup $S$ in the sense
of Definition \ref{pord},
it could be expanded to an infinitary  
semigroup in the sense of Definition \ref{pnc}
by setting all $\prod _{j \in J} a_ { \pi ^{-1} (j) }    $ as above
to be defined and equal to $\prod _{ \alpha  < \delta } a_\alpha  $.
 See Appendix I in \cite{smgrnc} for full details.

(b) In this note Definition \ref{pord} is enough for all 
purposes, we have recalled property (N)
for motivation and since it is somewhat simpler.
In any case, according the  comment in (a), we feel free to write,
say,
$\prod _{\beta_j \leq  \gamma <  \varepsilon'_j } a _{\gamma }$,
or even $\prod _{\pi(i) =j }  a_i$ 
in place of 
$\prod _{ \gamma < \varepsilon_j } a _{\beta_j + \gamma }$
in the situation from Clause (Ord), where $\varepsilon'_j =
\beta_j  +\varepsilon $.
Note that, if $\beta_j >0$, the interval   
$[\beta_j, \beta_j + \varepsilon_j )   $ 
is well-ordered, but, strictly speaking, not
an ordinal.

(c)
If in Definition \ref{pnc} $\prod _{i \in I} a_i $ is defined for every  
$I$ such that  $ \vert   I  \vert = 3  $,
then by (N)
we get  the structure of a semigroup in the classical sense.
In a classical semigroup all finite products are defined, hence,
up to remark (a) above, in the present terminology we have that 
a ${\leq} 3$-semigroup can be uniquely extended to a
${< }\omega $ semigroup (this amounts to prove the finite
generalized associative property).
Thus our main result Theorem \ref{ordg}   is an infinitary generalization
(with a much more involved proof!)
of the above remark.
\end{remark}

\section{The main extension theorem} \labbel{ext}

\begin{theorem} \labbel{ordg}
Let $\gamma$ be an infinite ordinal.

If $(S, \prod)$ is a partial ordinal semigroup such that 
 $\prod _{i < \delta   } a_i $ is defined exactly for every 
nonzero ordinal $\delta \leq \gamma $,   then the operation
$\prod$ can be uniquely expanded to an operation $\prod^*$
(without extending the set $S$) such that 
$\prod^* _{i < \delta   } a_i $ is defined exactly for every 
nonzero ordinal $\delta$ such that $ |\delta| = | \gamma |$
and in such a way that the operations agree when both defined.

More formally, in the terminology from 
Definition \ref{pord},
if $\mu =  \vert  \gamma \vert  ^+$ is the successor of 
the cardinality of $\gamma$,
then
any ${\leq} \gamma$-semigroup $S$ 
can be uniquely expanded 
to
 a ${<} \mu  $-semigroup 
in such a way that only new products are defined, no new element is added to $S$.
 \end{theorem} 

\begin{proof} 
Let $(S,   \prod^ \gamma )$ be a  ${\leq} \gamma $-semigroup.
It is enough to prove the following statement.

  \begin{enumerate}
 \item[(*)]  For every $\delta \geq \gamma $ such that 
$ \vert  \delta \vert   =  \vert  \gamma \vert  $,  $(S,   \prod^ \gamma )$ can be 
uniquely expanded 
to a 
${\leq} \delta $-semigroup $(S,   \prod^ \delta ) $.
  \end{enumerate}

Indeed, suppose that we can prove (*), $\delta \geq \varepsilon 
\geq \gamma $ and 
$ \vert  \delta \vert   =  \vert   \varepsilon  \vert   = \vert  \gamma \vert  $.
Applying (*) twice to both $\delta$ and $\varepsilon$ in place of $\delta$,
we get a ${\leq} \delta $-semigroup $(S,   \prod^ \delta )$
and a ${\leq} \varepsilon  $-semigroup $(S,   \prod^ \varepsilon )$.
If we restrict $ \prod^\delta $ to those products
which are 
on sequences of order-type $ \leq\varepsilon$,
 we get necessarily
 $\prod^ \varepsilon  $, by the uniqueness condition in  (*).
In other words, the sequence $( \prod^ \delta) _{ \delta <  \vert   \gamma  \vert  ^+}  $ is coherent,
hence we can join all the operations in order to get
 a  ${<}  \vert  \gamma \vert  ^+ $-semigroup $(S,   \prod^*)$.
Property (Ord) is satisfied in $(S,   \prod^*)$
since, given a sequence as in (Ord), the condition can be evaluated
already in $(S,   \prod^ \delta )$, since $\pi$ is assumed to 
be surjective, hence $\eta \leq \delta $ and also each
$ \varepsilon _j \leq \delta $.  

Satisfaction of Property (U) is straightforward
 here and in 
all the  arguments below.

So let us prove (*).
The proof goes by transfinite induction on those $\delta $'s 
such that   $\delta \geq \gamma $
and $ \vert  \delta \vert   =  \vert  \gamma \vert  $.
The proof will go through the next two
sections. 

The basis 
of the induction  $\delta = \gamma $ is straightforward:
$( S, \prod^ \gamma )$ itself is the only possible expansion
(defining no additional product) and by assumption it
is  a ${\leq} \gamma$-semigroup.

\section{Proof of the successor step} \labbel{succ} 

Suppose that $\delta = \varepsilon +1$.
Recall that an ordinal is the set of smaller ordinals, 
thus if $\delta= \varepsilon +1$, then $\delta = \{ \, \eta \mid
\eta \text{ an ordinal, } \eta \leq \varepsilon   \, \} $.
In particular, $\varepsilon$ is the maximal element of $\delta$.  
By the inductive hypothesis,  $( S, \prod^ \gamma )$ 
 can be uniquely expanded to a
 ${\leq} \varepsilon  $-semigroup $(S, \prod)$.
Here we write $\prod$ in place of $\prod ^ \varepsilon $ for notational simplicity.  
We want to extend $\prod$ to an operation $\prod'$ which is 
always defined on  $\delta$-indexed sequences of elements of $S$.
If $( a_i) _{i < \delta } $ is a sequence of elements from $S$, 
and  (U) from Definition \ref{pnc},
(Ord) from Definition \ref{pord} 
have to be satisfied, then   necessarily
\begin{equation}\labbel{e}     
\sideset{}{'}\prod_{i < \delta  } a_i = 
\left( \prod _{i < \varepsilon  } a_i \right)  * a_ \varepsilon  
 \end{equation}
This is shown to be necessary by taking
$\pi: \delta \to \{ 0,1\} $ such that 
$ \pi(i)=0$, for every  $i < \varepsilon $, and
$\pi( \varepsilon  )= 1$.
Note that  $\prod _{i < \varepsilon  } a_i$ is defined,
 since,  by assumption, $(S, \prod)$ is a ${\leq} \varepsilon  $-semigroup.
Moreover the binary operation $*$ is defined already in $( S, \prod^ \gamma )$, since 
we have assumed that $\gamma$ is infinite. Hence if 
$(S, \prod)$ can be expanded to a 
${\leq} \delta   $-semigroup $(S, \prod')$, then
the expansion is unique. 

 Since, by the inductive assumption, $(S, \prod)$
 is unique with respect to $( S, \prod^ \gamma )$, then 
$(S, \prod')$, if actually a
${\leq} \delta   $-semigroup, is also the unique expansion
 of $( S, \prod^ \gamma )$.  

We now take equation \eqref{e}
as a definition, simultaneously for every 
$\delta$-indexed sequence.
Of course, we let $\prod'$ coincide with 
$\prod$ for sequences indexed by nonzero ordinals $\leq \varepsilon $. 
We have to show that (Ord) is satisfied.
By the inductive hypothesis,
 (Ord) is satisfied for sequences of length $\leq \varepsilon$.
Note that, in particular, we have that
\begin{equation}\labbel{e2}     
\sideset{}{'}\prod _{i < \eta } a_i = 
\left(  \prod _{i < \zeta  } a_i \right)  * a_ \zeta  
 \end{equation}
whenever $\eta = \zeta +1 \leq \delta $. 
This follows from equation \eqref{e}, which is now a definition,
if $\eta= \delta $, and, if   $\eta < \delta $, 
it follows from the same argument
justifying \eqref{e}, since 
(Ord) holds in $(S, \prod)$ by
the inductive hypothesis.

So let $( a_i) _{i < \delta } $
be a
sequence of elements of $S$ and
 $\pi: \delta \to \eta $ 
be an order-preserving  surjection. 
Since $\pi$ is  order preserving
and $\delta$ has a maximum, 
$\eta$ has a maximum, too, thus $\eta= \zeta +1$,
for some ordinal $\zeta$. The proof of the successor step
is now split into two cases. 

\subsection{Case (a)} \labbel{subsub} 
 First suppose that $ \pi ^{-1}( \{ \zeta \} ) = \{ \varepsilon \}  $
and 
let $\pip : \varepsilon \to \zeta $ be
the restriction of $\pi$ to $\varepsilon$.
Note that 
$\pip$ is surjective since, by assumption,
$\pi$ is surjective, and  $ \pi ^{-1}( \{ \zeta \} ) = \{ \varepsilon \}  $.
For $j < \eta $, let 
$b_j=\prod _{\pi(i) =j }  a_i$,
where we have used the conventions from
Definition \ref{pnc} and Remark \ref{strfrm}(b).
Note that, as already observed in a similar situation,
 for each $j < \eta $, the set $\{ \, i < \delta  \mid \pi(i) =j  \,\}$ 
has order type $\leq \varepsilon $, hence 
$\prod _{\pi(i) =j } a_i$ is defined.

By our assumption and
property (U), we have that
$b_ \zeta = a_ \varepsilon $ and,
for $j < \zeta $, we also have
$b_j=\prod _{\pip(i) =j }  a_i $,
since, by construction,  if $j < \zeta $,
then $\{ i < \delta \mid \pip(i) =j\} = \{ i < \delta \mid \pi(i) =j\}$. 
 We perform the following computations, where, for convenience, 
we denote  by $=^{\text{def}} $ an identity which follows from
the above positions and considerations. We denote by $= ^{ \eqref{e}}$, 
$= ^{\text{(Ord)}}$, etc.,
 an identity which follows from the corresponding equations, 
true by the above remarks, provided the index sets are checked to
satisfy the appropriate hypotheses.
\begin{gather}   
\sideset{}{'}\prod _{i < \delta  } a_i 
= ^{ \eqref{e} } 
\left( \prod _{i < \varepsilon  } a_i \right)  * a_ \varepsilon  
= ^{\text{(Ord)}}
\left(
  \prod _{j < \zeta } \prod _{\pip(i) =j }  a_i 
\right)
 * a_ \varepsilon 
=^{\text{def}} 
\left(  \prod _{j < \zeta } b_j \right) 
 * b_ \zeta 
 \\
\labbel{4}
\sideset{}{'}\prod _{j < \eta } 
  \sideset{}{'}\prod _{\gamma <  \varepsilon _j }  a _{ \beta _j + \gamma }  
=^{\text{def}}  
 \sideset{}{'}\prod _{j < \eta } \prod _{\pi(i) =j }  a_i 
=^{\text{def}} 
  \sideset{}{'}\prod _{j < \eta }b_j
= ^{\eqref{e2}} 
\left( \prod _{j < \zeta } b_j \right)
 * b_ \zeta 
 \end{gather}
where, as in Definition \ref{pord}, 
we have set $ \{ i < \delta \mid \pi ( i) = j \}
= [\beta_j, \beta_j + \varepsilon_j )  $. 

In the last identity in \eqref{4} 
we have to resort to \eqref{e2}, since
both the possibilities
$\eta= \delta $ and
  $\eta <  \delta $ might occur.
In any case, $\prod' _{i < \delta  } a_i  = 
  \prod' _{j < \eta } \prod' _{\gamma <  \varepsilon _j }  a _{ \beta _j + \gamma }  $
is proved   in case (a).

\subsection{Case (b)} \labbel{caseb1}  Suppose that case (a) fails, that is, 
$ \pi ^{-1}( \{ \zeta  \}) \supset  \{ \varepsilon \}  $ properly. Recall that 
$\delta = \varepsilon +1$, 
$\eta= \zeta +1$
and that 
 $\pi$ is  surjective and  order preserving
from $\delta$ to $\eta$.
Let again $\pip$ be the restriction 
of $\pi$ to $\varepsilon$. Note that in the present case
the image of $\pip$ is the whole of $\eta$.  
In particular, $\eta \leq  \varepsilon $,
hence a product of an $\eta$-indexed sequence is 
$ \prod$-defined (this is a shorthand to mean that the product is defined according 
to $\prod$).    

For $j < \eta $, let 
$b_j=\prod _{\pip(i) =j }  a_i $ and
$b'_j=\prod' _{\pi(i) =j }  a_i $. Recall that, in view of Remark \ref{strfrm}, 
here and below we are free to write $\prod' _{\pi(i) =j }  a_i $
in place of the more formal 
$\prod' _{\gamma <  \varepsilon _j }  a _{ \beta _j + \gamma }$,
for appropriate $\varepsilon _j$ and $\beta_j$.  

As in case (a), if $j < \zeta$ 
then $b_j=b'_j$. 
On the other hand, notice that $i=\varepsilon$ is not 
in the domain of $\pip$, hence  the factor $a_ \varepsilon$
 does not appear
in $b_ \zeta=\prod _{\pip(i) = \zeta }  a_i $, while it appears in 
$b'_ \zeta=\prod' _{\pi(i) = \zeta }  a_i $.
Hence it is possible, indeed  likely, that 
$b_ \zeta \not= b'_ \zeta$.
Moreover, it is possible that
$ K= \{ i < \delta \mid \pi(i) =\zeta\}$ has order-type
$\delta$. If this is case, the product
defining  $b'_ \zeta $, that is, 
$\prod' _{i \in K }  a_i $,
is not an operation of  $(S, \prod)$, but it is 
just a special case of the definition given by equation \eqref{e}.
It is important to notice that in the computations below, 
though we will use the fact that $b'_ \zeta $ is always a well-defined
element of $S$, we will \emph{not} use the assumption 
(which might be false) that 
$\prod _{i \in K }  a_i $ is defined.
Finally, notice that 
$\varepsilon$ is the maximum of $K$,
whence
\begin{equation}\labbel{dubal}      
b'_ \zeta =^{\text{def}} 
\sideset{}{'}\prod _{i \in K }  a_i = ^{\eqref{e2}}
\left( \prod _{i \in K \setminus \{ \varepsilon  \} }  a_i \right)  * a_ \varepsilon
=^{\text{def}}
  b_ \zeta  *  a_ \varepsilon
  \end{equation}
With the above provisions, we now perform the following computations.

\begin{multline*}     
\sideset{}{'} \prod _{i < \delta  } a_i 
= ^{ \eqref{e} } 
\left( \prod _{i < \varepsilon  } a_i \right)  * a_ \varepsilon  
= ^{\text{(Ord)}}
\left(
  \prod _{j < \eta } \prod _{\pip(i) =j }  a_i 
\right)
 * a_ \varepsilon 
=^{\text{def}} 
\left(  \prod _{j < \eta } b_j \right) 
 *  a_ \varepsilon 
 \\  
= ^{\text{(Ord)}}
 \left( \left( 
 \prod _{j < \zeta } b_j \right)  * b_ \zeta 
\right)
 *  a_ \varepsilon 
= ^{\text{(Ord)}}
\left(  \prod _{j < \zeta } b_j \right) * 
( b_ \zeta  *  a_ \varepsilon )
\\
=^{\text{def},\eqref{dubal}}
\left(  \prod _{j < \zeta } b'_j \right) * 
 b'_ \zeta  =^{\eqref{e2}}
\sideset{}{'}\prod_{i < \eta   } b'_j =^{\text{def}}
 \sideset{}{'}\prod _{j < \eta } \sideset{}{'}\prod _{\pi(i) =j }  a_i
\end{multline*}

\section{Proof of the limit step} \labbel{lim} 
\subsection{Equivalent definitions} \labbel{eqd} 

Now assume that $\delta$ is a limit ordinal. 
By the inductive hypothesis,
for every $\varepsilon < \delta $,
$(S, \prod^ \gamma )$  can be expanded uniquely 
to a
 ${\leq} \varepsilon   $-semigroup $(S, \prod^ \varepsilon )$.
By the same  coherence argument exploited in Section \ref{ext}
and by the uniqueness of each  $(S, \prod^ \varepsilon )$, we have that
$(S, \prod^ \gamma )$  can be expanded uniquely 
to a
${<} \delta   $-semigroup $(S, \prod)$.
We need to expand $\prod$ to an operation $\prod'$ which applies also
to sequences of order-type $ \delta$.  
 Since $ \vert  \delta \vert   \leq \gamma $, we have  
$\cf\delta \leq \gamma $. Fix a sequence $(\beta_j) _{j < \cf \delta } $
 satisfying the following property: 
  \begin{enumerate}
    \item[(P$_{ \delta }$)]
The sequence 
$(\beta_j) _{j < \cf \delta  } $ is closed, strictly increasing and cofinal 
in $\delta$. Moreover $\beta_0=0$.
  \end{enumerate}

In presence of a sequence as above, we will introduce the following notation.
\begin{enumerate}
    \item[(D$_{ \delta }$)]
Let $\lambda =\cf\delta  $.
For every $i < \delta $,
there is a unique $j_i  < \lambda$  
  such that  
$i \in I _{j_i}= [ \beta _{j_i},  \beta _{j _{i}+1 }) $. 
Let $p: \delta \to \lambda $ be defined by
setting $p(i )= j_i $.
  \end{enumerate}
  
If $( a_i) _{i < \delta } $ is a sequence of elements of  $S$, 
and (Ord) from Definition \ref{pord}  
has to be satisfied, then we must have
\begin{equation}\labbel{f}     
\sideset{}{'}\prod _{i < \delta  } a_i =
 \prod _{j < \lambda  } \prod 
_{i \in I_j}
 a_i   
 \end{equation}
since $I_j = \{ i \in I \mid p(i) =j \} $. 
All products on the right-hand side
of equation \eqref{f} are defined,
since 
$(S, \prod)$
 is a ${<} \delta   $-semigroup,
$ \lambda = \cf\delta \leq  \gamma < \delta $ and 
$I _{j}= [ \beta _{j},  \beta _{j+1 }) \subseteq \beta _{j+1 } $
has order-type $\leq \beta _{j+1 } < \delta $.
Hence if $(S, \prod)$ can be expanded to 
a ${\leq} \delta   $-semigroup,
then the expansion has to satisfy equation \eqref{f},
hence it is unique. Note for later use
that,
 by (Ord), equation \eqref{f} is satisfied, under the corresponding assumptions,
also for sequences indexed by some limit ordinal $\delta' < \delta $
and with the product $\prod$ everywhere. 

We can use \eqref{f} in order 
to define  
$\prod '_{i < \delta  } a_i $
for every sequence $( a_i) _{i < \delta } $
of elements of $S$.
Though the definition is relative to the  sequence
  $(\beta_j) _{j < \cf \delta } $ that we have chosen,
we will show in the following claim that the definition is independent from
the sequence.

\begin{claim} \labbel{claim}    
 The outcome of the right-hand
side of \eqref{f} does not depend on the choice of 
the sequence $(\beta_j) _{j < \cf \delta  } $
provided that the sequence   satisfies (P$_ \delta $).
In particular,  we can use equation \eqref{f} as a definition
independent from the chosen sequence.
\end{claim}

To prove the claim, 
let 
  $(\beta''_j) _{j < \cf \delta  } $
be another sequence which satisfies (P$_ \delta $).
If we consider the set-theoretical union
  $B =\{ \beta_j \mid  j < \lambda \} \cup  \{ \beta''_j \mid  j < \lambda \}  $,
then we can arrange 
the elements of $B$ in increasing order,
to form another sequence
$(\beta'_j) _{j < \cf \delta } $
which still satisfies (P$_ \delta $).
The only nontrivial property to be checked is that 
$B$ has actually order-type $\cf\delta$.
Since   $\cf\delta = \lambda $ is an infinite cardinal, 
each initial segment of both 
$\{ \beta_j \mid  j < \lambda \}$ and $\{ \beta''_j \mid  j < \lambda \}$
has order-type $<\lambda$.
Since    $(\beta''_j) _{j < \cf \delta  } $ and
 $(\beta_j) _{j < \cf \delta  } $
are mutually cofinal,  
each initial segment of 
$B$ has order-type $<\lambda$,
hence 
$B$ has order-type $\lambda = \cf \delta $, 
hence the index set of 
 $(\beta'_j) _{j < \cf \delta } $ is the appropriate one.

Thus  $(\beta_j) _{j < \cf \delta } $  is a \emph{subsequence}
of  $(\beta'_j) _{j < \cf \delta } $ in the sense 
that every element  of the first sequence appears in the second sequence,
possibly with a different index.
If we show that equation \eqref{f} gives the same outcome
when applied to some subsequence (in the above sense)
of a given sequence, then  \eqref{f} gives the same outcome
in the case of both $(\beta_j) _{j < \cf \delta } $
and  $(\beta'_j) _{j < \cf \delta } $ as above.
By the same result, the outcome is the same relative to
$(\beta''_j) _{j < \cf \delta } $ and to $(\beta'_j) _{j < \cf \delta } $,
and so it is the same relative to $(\beta_j) _{j < \cf \delta } $
 and to $(\beta''_j) _{j < \cf \delta } $.
Thus, in order to prove the claim, 
it is enough to prove it in the particular case of two sequences
such that 
one  is a subsequence of the other.

So suppose that
$(\beta_k) _{k < \cf \delta } $  is a subsequence
of  $(\beta'_j) _{j < \cf \delta } $ and that
both sequences satisfy
(P$_ \delta $).
Let 
$j_i$, $I_j = [ \beta _j, \beta _{j+1} )$ 
and $p$ be defined as in (D$_ \delta $),
relative to the sequence $(\beta_j) _{j < \cf \delta } $.  
Similarly, relative to the sequence
$(\beta'_k) _{k < \cf \delta } $ and 
for every 
 $i < \delta $,
let  $k_i  $ 
 be the unique element of $  \lambda$  such that  
$i \in K _{k_i}= [ \beta' _{k_i},  \beta' _{k _{i}+1 }) $.
Let $q: \delta \to \lambda $ be defined by
$q(i )= k_i $.  
Since $(\beta_k) _{k < \cf \delta } $  is a subsequence
of  $(\beta'_j) _{j < \cf \delta } $,
 for every $k < \lambda$
there is some 
$j(k) < \lambda$  
such that $K_k \subseteq I _{ j(k)} $.
Moreover, for every $j < \lambda$,  
 $\{ K_k \mid j(k) = j\}$ 
is a partition of $I_j$. 
Let $\pi: \lambda \to \lambda $
be defined by 
$\pi ( \kappa )= j(k)$.  

Let $( a_i) _{i < \delta } $ be a sequence  
of elements of $S$.
The above considerations, together with property (Ord)
in $(S, \prod)$, imply that,
for every $j < \lambda$, we have
$\prod _{i \in I_j  } a_i =
 \prod _{ \pi ( \kappa )= j } \prod _{i \in K_k} a_i   $.
Letting 
$b_k= \prod _{i \in K_k} a_i$
the above formula becomes 
$\prod _{i \in I_j  } a_i =
 \prod _{ \pi ( \kappa )= j } b_k $
and then we get the following identities.
\begin{equation}\labbel{cac}
 \prod _{j < \lambda  } \prod _{i \in I_j} a_i   
=
 \prod _{j < \lambda  }  \prod _{ \pi ( \kappa )= j } b_k  
= ^{\text{(Ord)}}
 \prod _{k < \lambda  }  b_k
=^{\text{def}}
 \prod _{k < \lambda  }  \prod _{i \in K_k} a_i  
  \end{equation}    

Note that all the above subproducts and products are 
defined in $(S, \prod)$, since all the sequences under consideration
have order-type
$<\delta$,
hence we can actually apply (Ord). 
 The left-hand side in equation \eqref{cac}
is the outcome of \eqref{f}
as computed with respect to  
$(\beta_j) _{j < \cf \delta } $ and the 
right-hand side in \eqref{cac}
is the outcome of \eqref{f}
as computed with respect to  
$(\beta'_j) _{j < \cf \delta } $.
The identity in \eqref{cac} thus proves the claim. \qed

Armed with the above claim, we are entitled to use equation \eqref{f}
as a definition for a structure $(S, \prod')$ expanding $(S, \prod)$ by a new 
infinitary operation on $\delta$-ordered sequences. 
It remains to check
that $(S, \prod')$  actually satisfies property (Ord)
from Definition \ref{pord}. Recall that
 we write, say, $\prod' _{\pi(i) =j }  a_i $
for  something like 
$\prod' _{\gamma <  \varepsilon _j }  a _{ \beta _j + \gamma }$.
 
So let $\pi: \delta \to \eta $, for some $\eta$,  
be a surjective order preserving map.
The proof will be again divided into cases.

\subsection{Case (a)} \labbel{IIa} 
 The ordinal $\eta$ is successor.
Say, $\eta= \zeta +1$.
Let $\beta= \inf \pi^{-1}( \{ \zeta \} ) $.
Thus $\eta$ is the image of $\beta+1$ under $\pi$ and,
since $\pi$ is surjective, we have  $\eta \leq \beta +1 < \delta $,
since $\delta$ is limit.
Let  $(\beta_j) _{j < \cf \delta } $ be a sequence 
of elements in $\delta$ satisfying 
(P$_ \delta $). Without loss of generality, by Claim \ref{claim},
we can suppose that $\beta$ is an element of the sequence, 
say, $\beta= \beta _{\bar\jmath} $.
Let 
$j_i$, $I_j = [ \beta _j, \beta _{j+1} )$ 
and $p$ be defined as in  
 (D$_ \delta $), and let $\lambda= \cf \delta $.

We will now introduce some auxiliary elements of $S$.
For $h < \eta $, let $d_h =\prod' _{ \pi(i)= h} a_i   $. 
If $h <\zeta$, then the above product defining $d_h$  is $\prod$-defined,
 since $\{ i < \delta \mid \pi(i)= h \}$ 
has order-type $<\delta$. 
 However, $d_\zeta$ 
might be or might not be $\prod$-defined,
 since   $\{ i < \delta \mid \pi(i)= h \}$
might have order-type  $<\delta$ or
order-type  $\delta$.
In the former case 
$d_\zeta$  is $\prod$-defined;
in the latter case, 
the definition of  
$d_\zeta$  is given by equation \eqref{f}.
As in Subsection \ref{caseb1}, 
we will use the existence of the (unique and well-defined, by Claim \ref{claim}) element  
$d_\zeta$, but we will never use the (possibly false) assumption 
that $d_\zeta$ is given by a $\prod$-defined product.

As the reader has surely already become accustomed to by now,
  the proof proceeds by a number of computations.
 The first computation involves only $\prod$.
\begin{equation}\labbel{7}
\prod _{h < \zeta   } d_h
=^{\text{def}}
 \prod _{h < \zeta   } \prod _{ \pip(i)= h} a_i
= ^{\text{(Ord)}}
 \prod _{i < \beta } a_i   
= ^{\text{(Ord)}}
 \prod _{j < \bar\jmath  } \prod _{i \in I_j} a_i   
  \end{equation}
where in the first application of (Ord)
$\pip$ is the restriction of $\pi$ to $\beta$,  
and the second application of 
(Ord) is obtained
by considering the restriction of $p$ to $\beta$, 
using the fact that 
$p(\beta) = \bar\jmath$ and $\beta$ is minimal with this property.
Note that $i \in I_j$ 
if and only if 
$p(i)=j$. 

We now compute 
$d_ \zeta  =\prod' _{ \pi(i)= \zeta } a_i =
 \prod' _{ \beta \leq i < \delta } a_i  =
 \prod' _{ k < \bar\delta } a _{ \beta +k} $,
where $\bar\delta$ is the only ordinal such that 
$ \beta + \bar\delta = \delta $.
Here and in what follows we are using the fact that
$[ \beta , \delta )$ is well-ordered, hence order isomorphic
to some ordinal, precisely the above-defined $\bar\delta$.
It might well happen that (i) $\bar\delta = \delta $; if this is the case, then
$d_ \zeta $ is given by definition  \eqref{f}.
On the other hand, if (ii) $\bar\delta < \delta $,
then    $d_ \zeta $ is given  by a product defined in $(S, \prod)$.
In case (i) it is necessary  to choose a sequence 
satisfying (P$_ \delta $); by Claim \ref{claim},
we can choose a sequence at will.
For every $ \ell  <  \lambda = \cf \delta = \cf\bar\delta$, 
let $\bar\beta_ \ell $
be the only ordinal such that 
$ \beta _{\bar\jmath + \ell} = \beta  + \bar\beta_ \ell$. 
Since, by assumption, $ \beta = \beta _{\bar\jmath}$, we have 
$\bar\beta_0 = 0$. 
By using the corresponding properties
of the sequence 
 $(\beta_j) _{j < \cf \delta } $, it is easy to check
that 
 $(\bar\beta_\ell) _{\ell < \cf \delta } $
satisfies (P$_{\bar \delta} $),
whether $\bar\delta = \delta $ or not.
Hence $p: \delta \to \lambda $
induces a surjection
 $\bar p: \bar\delta \to \lambda $
defined by letting $\bar p( k)$, for 
$k < \delta$,
to be the unique 
$ \ell$ such that   
$\beta+ k \in [\beta _{\bar\jmath + \ell}, \beta _{\bar\jmath + \ell+1})$. 
 
Hence we can apply  
\eqref{f} in both cases, to get the second identity in the next equation,
setting $\bar I_ \ell 
= \{k < \bar \delta  \mid  \bar p ( k) = \ell  \} =
\{k < \bar \delta  \mid  p ( \beta + k) =  \bar\jmath + \ell \} $
and then noticing that
$j \geq \bar\jmath $ if and only if there is
some $\ell < \lambda $ such that $j=  \bar\jmath + \ell$.
\begin{equation}\labbel{8}    
 d_ \zeta
=^{\text{def}}
\sideset{}{'} \prod _{ k < \bar\delta } a _{ \beta +k} 
=^{ \eqref{f}}
 \prod _{ \ell < \lambda  } \prod _{k\in \bar I_ \ell}  a _{ \beta +k}
=^{\text{def}} 
 \prod _{j \geq \bar\jmath } \prod _{i \in I_j} a_i  
  \end{equation}
For $j < \lambda$, let $c_j =\prod _{i \in I_j} a_i   $. 
Note that the product is defined,
since  $[ \beta _j, \beta _{j+1} )$ has order-type 
$\leq \beta _{j+1}< \delta  $. 
 Now for our final computations
in case (a).
\begin{gather*}     
\sideset{}{'}\prod _{i < \delta  } a_i 
\stackrel{\eqref{f}}{=}   
 \prod _{j < \lambda  } \prod _{i \in I_j} a_i   
\stackrel{\text{def}}{=}   
 \prod _{j < \lambda  } c_j
\stackrel{\text{(Ord)}}{=} 
\prod _{j < \bar\jmath  } c_j *  \prod _{j \geq \bar\jmath } c_j
\stackrel{\text{def}}{=}
 \prod _{j < \bar\jmath  } \prod _{i \in I_j} a_i *  \prod _{j \geq \bar\jmath } \prod _{i \in I_j} a_i  
\\  
 \prod _{h< \eta } \sideset{}{'}\prod _{\pi(i) =h }  a_i
\stackrel{\text{def}}{=}
 \prod _{ h< \eta } d_h
\stackrel{\text{(Ord)}}{=} 
\left(  \prod _{h < \zeta   } d_h \right) 
*
 d_ \zeta 
=^{ (\ref{7}, \ref{8})} 
\prod _{j < \bar\jmath  } \prod _{i \in I_j} a_i *  \prod _{j \geq \bar\jmath } \prod _{i \in I_j} a_i  
\end{gather*}

The above equations show that $(S, \prod')$  satisfies (Ord) in case (a). 

\subsection{Case (b)} \labbel{IIb} 
 The ordinal  $\eta$ is  limit.
Recall that $\pi: \delta \to \eta $ is  surjective
and order preserving.
Let $( \zeta _j) _{j < \cf \eta  } $
be a sequence satisfying (P$_{ \eta } $).
For every $j< \cf \eta $,
let 
$ \beta_j = \inf \pi ^{-1}( \{ \zeta _j \} ) $.
Since $\pi: \delta \to \eta $ is surjective,  
$( \beta _j) _{j <  \cf \eta  } $
is cofinal in $\delta$,
hence $ \cf \delta = \cf \eta $.
It is standard to see that the assumption that 
$( \zeta _j) _{j < \cf \eta   } $
 satisfies (P$_{ \eta } $)
implies that 
$( \beta _j) _{j <  \cf \delta   } $
satisfies (P$_{ \delta  } $).
As usual, let 
$j_i$, $I_j = [ \beta _j, \beta _{j+1} )$ 
and $p$ be defined by
 (D$_ \delta $), relative to the sequence 
$( \beta _j) _{j <  \cf \delta  } $.
By Claim \ref{claim}, we can use
$( \beta _j) _{j <  \cf \delta   } $ in equation \eqref{f} 
 in order to evaluate $\prod' _{i < \delta  } a_i $.
Before performing the final computation, for every $j <  \cf \delta $
define $\pi_j$ to be the restriction of $\pi$ 
to $I_j = [ \beta _j, \beta _{j+1} )$, thus
$\pi_j$ is surjective and  order preserving from 
 $I_j$ to  $J_j=[ \zeta _j, \zeta _{j+1} )$ and,
for $h \in J_j$, we have 
$\{ i < \delta \mid \pi_j(i) =h\} = \{ i < \delta \mid \pi(i) =h\}$.
Setting $b_h=  \prod _{\pi(i)=h} a_i $, for
$h < \eta$, we get 
\begin{equation}\labbel{dubb}      
\prod _{i \in I_j} a_i = ^{\text{(Ord)}} \prod _{h \in J_j} \prod _{\pi(i)=h} a_i 
=^{\text{def}} \prod _{h \in J_j} b_h 
  \end{equation}
for every $j < \lambda $. The next computation finishes
the proof of Case (b) and of  Theorem \ref{ordg}, as well.
\begin{equation*}     
\sideset{}{'}\prod _{i < \delta  } a_i 
=^{ \eqref{f}} 
 \prod _{j < \cf \eta   } \prod _{i \in I_j} a_i   
=^{ \eqref{dubb}} 
\prod _{j < \cf \eta   }  \prod _{h \in J_j} b_h 
= ^{ \eqref{f}}
\sideset{}{'}\prod _{h < \eta } b_h 
=^{\text{def}}
\sideset{}{'}\prod _{h < \eta } \prod _{\pi(i)=h} a_i \qedhere 
 \end{equation*}
\end{proof}

\section{Further remarks} \labbel{fur} 

\begin{remark} \labbel{nonsi}
(a)   We cannot improve the conclusion 
in the last statement in Theorem \ref{ordg} 
to an expansion which is a ${\leq} \mu$-semigroup.
For example, let $\gamma= \omega $, thus $\mu = \omega_1 $
and let $S$ be the set of (finite or) countable ordinals with 
the ordinal sum restricted to sequences indexed by a countable ordinal.
If it is possible to expand the operation
\emph{without adding further elements}, then, say,
$\sum _{i < \omega_1} 1 $ should have a definite value which must
be a countable ordinal, say, $\delta$. Let $\varepsilon$ be a countable
ordinal $> \delta $. Then
$ \delta = \sum _{i < \omega_1} 1 =^{(Ord)} \left(\sum _{i < \varepsilon } 1\right)
+\sum _{ \varepsilon \leq i < \omega_1} 1 = \varepsilon + \delta > \delta  $,
since we have chosen $\varepsilon  > \delta$ and since the interval
$[ \varepsilon , \omega_1)$ has order type $ \omega _1$.   
We have reached a contradiction, hence an expansion as above
does not exist. 
   
(b) If, instead, we allow the possibility of adding 
further elements to the set $S$, then every ${\leq} \gamma $-semigroup
can be extended to an ordinal semigroup, actually, to
a complete infinitary semigroup. First, expand $S$ to a
${<} \mu  $-semigroup in view of Theorem \ref{ordg},
where $\mu =  \vert  \gamma \vert  ^+$.
Then add a fully absorbing element $ \Omega$ 
and set to $ \Omega$ all those  products with index $\geq \mu$ 
(and all those products indexed by a non well-ordered set).  
See \cite[Example 5.4]{smgrnc} for full details. 
\end{remark}

In (b) above we actually need the fact that $\mu $ 
is a regular cardinal.

\begin{proposition} \labbel{singsing}
If $\mu $ is a singular cardinal, then
every ${<}\mu $-semigroup can be uniquely expanded
to a   ${\leq}\mu $-semigroup, hence, by Theorem \ref{ordg},
 to a ${<}\mu^+ $-semigroup. 
\end{proposition} 

 \begin{proof} 
By taking $\delta= \mu $ in \eqref{f} 
and using exactly the same arguments  in the proof of
the limit step of Theorem \ref{ordg}, as presented in Section \ref{lim}.   
\end{proof}

\end{document}